\documentclass[12pt]{article}
\usepackage{geometry,mathtools}
\usepackage{amsmath,amssymb,amsthm, amsfonts}
\usepackage{etaremune, enumerate, float, verbatim}
\usepackage{algorithm}
\usepackage{algorithmic}
\usepackage{hyperref}
\usepackage{graphicx}
\usepackage{url}
\usepackage[mathlines]{lineno}
\usepackage{dsfont} 
\usepackage{tikz, graphicx, subcaption, caption}
\usepackage[algo2e,ruled,vlined]{algorithm2e}
\usepackage{mathrsfs,amsmath}
\usepackage{color}
\definecolor{red}{rgb}{1,0,0}

\definecolor{blue}{rgb}{0,0,.7}

\definecolor{green}{rgb}{0,.6,0}

\definecolor{purp}{rgb}{.5,0,.5}

\numberwithin{figure}{section}   

\newtheorem{thm}{Theorem}[section]

\newtheorem{lem}[thm]{Lemma}

\theoremstyle{definition}

\theoremstyle{definition}

\theoremstyle{definition}

\newcommand{\exm}{\operatorname{ex}}

\newcommand{\sat}{\operatorname{sat}}

\newcommand{\bit}{\begin{itemize}}
\newcommand{\eit}{\end{itemize}}
\newcommand{\ben}{\begin{enumerate}}
\newcommand{\een}{\end{enumerate}}
\newcommand{\beq}{\begin{equation}}
\newcommand{\eeq}{\end{equation}}
\newcommand{\bea}{\begin{eqnarray*}} 
\newcommand{\eea}{\end{eqnarray*}}
\newcommand{\bpf}{\begin{proof}}
\newcommand{\epf}{\end{proof}\ms}
\newcommand{\bmt}{\begin{bmatrix}}
\newcommand{\emt}{\end{bmatrix}}
\newcommand{\ms}{\medskip}

\newcommand{\noi}{\noindent}

\title{Almost all permutation matrices have bounded saturation functions}
\author{Jesse Geneson}

\begin{document}
\maketitle

\begin{abstract}
Saturation problems for forbidden graphs have been a popular area of research for many decades, and recently Brualdi and Cao initiated the study of a saturation problem for 0-1 matrices. We say that 0-1 matrix $A$ is saturating for the forbidden 0-1 matrix $P$ if $A$ avoids $P$ but changing any zero to a one in $A$ creates a copy of $P$. Define $\sat(n, P)$ to be the minimum possible number of ones in an $n \times n$ 0-1 matrix that is saturating for $P$. Fulek and Keszegh proved that for every 0-1 matrix $P$, either $\sat(n, P) = O(1)$ or $\sat(n, P) = \Theta(n)$. They found two 0-1 matrices $P$ for which $\sat(n, P) = O(1)$, as well as infinite families of 0-1 matrices $P$ for which $\sat(n, P) = \Theta(n)$. Their results imply that $\sat(n, P) = \Theta(n)$ for almost all $k \times k$ 0-1 matrices $P$.

Fulek and Keszegh conjectured that there are many more 0-1 matrices $P$ such that $\sat(n, P) = O(1)$ besides the ones they found, and they asked for a characterization of all permutation matrices $P$ such that $\sat(n, P) = O(1)$. We affirm their conjecture by proving that almost all $k \times k$ permutation matrices $P$ have $\sat(n, P) = O(1)$. We also make progress on the characterization problem, since our proof of the main result exhibits a family of permutation matrices with bounded saturation functions.
\end{abstract}

\noi {\bf Keywords} saturation functions, extremal functions, pattern avoidance, 0-1 matrices

\noi{\bf AMS subject classification} 05D99 \bigskip

\section{Introduction}
We say that 0-1 matrix $A$ \emph{contains} the 0-1 matrix $P$ if some submatrix of $A$ is either equal to $P$ or can be turned into $P$ by changing some ones to zeroes. Otherwise, we say that $A$ \emph{avoids} $P$. The extremal function $\exm(n, P)$ is defined as the maximum number of ones in an $n \times n$ 0-1 matrix that avoids $P$. This extremal function has been applied to a wide range of problems including the proof of the Stanley-Wilf conjecture \cite{mt}, the best known bounds on the maximum number of unit distances in a convex $n$-gon \cite{furedi}, a bound on the complexity of an algorithm for finding a minimal path in a rectlilinear grid with obstacles \cite{mitchell}, and bounds on extremal functions of forbidden sequences \cite{pettie}. 

There is a long history of the study of extremal functions of forbidden 0-1 matrices, see e.g. \cite{ts, kst, bg, pt, h, cm, gst}. Marcus and Tardos \cite{mt} proved that $\exm(n, P) = O(n)$ for every permutation matrix $P$, and used this fact to resolve the F\"{u}redi-Hajnal conjecture \cite{fh}, which also resolved the Stanley-Wilf conjecture using an earlier result of Klazar \cite{k}. In a different paper, Tardos finished bounding the extremal functions of all forbidden 0-1 matrices with at most four ones up to a constant factor \cite{T}. Many connections have been found between extremal functions of forbidden 0-1 matrices and the maximum possible lengths of generalized Davenport-Schinzel sequences \cite{ck0, cds, ff0, N, pettieds, wp}. 

Besides permutation matrices, other families of forbidden 0-1 matrices with linear extremal functions have been found including 0-1 matrices obtained by doubling permutation matrices \cite{g}, 0-1 matrices corresponding to visibility graphs \cite{fulek, gs}, and 0-1 matrices obtained by performing operations on smaller 0-1 matrices  with linear extremal functions \cite{keszegh, pettie}. After the result of Marcus and Tardos, Fox \cite{fox} proved that almost all $k \times k$ permutation matrices have $\exm(n, P) = 2^{\Omega(\sqrt{k})}n$. Fox also showed that $\exm(n, P) = 2^{O(k)}n$ for every $k \times k$ permutation matrix $P$, and these results for permutation matrices were also generalized to extremal functions of forbidden $d$-dimensional permutation matrices \cite{km, gt, ck}.

Saturation problems have been studied for graphs \cite{dkm, ehm, fk0, kt}, posets \cite{fkk, klm}, and set systems \cite{fkk0, gkl}. Recently Brualdi and Cao initiated the study of a saturation problem for 0-1 matrices \cite{bc}, which is different from the saturation problem for 0-1 matrices defined in \cite{dpt}. We say that 0-1 matrix $A$ is \emph{saturating} for $P$ if $A$ avoids $P$ but changing any zero to a one in $A$ creates a copy of $P$. Define the saturation function $\sat(n, P)$ to be the minimum possible number of ones in an $n \times n$ 0-1 matrix that is saturating for $P$. Brualdi and Cao proved for $n \geq k$ that $\exm(n, P) = \sat(n, P)$ for all $k \times k$ identity matrices $P$. 

Fulek and Keszegh \cite{fk} found a general upper bound on the saturation function in terms of the dimensions of $P$, and proved for every 0-1 matrix $P$ that either $\sat(n, P) = O(1)$ or $\sat(n, P) = \Theta(n)$. They found two 0-1 matrices $P$ for which $\sat(n, P) = O(1)$, the $1 \times 1$ identity matrix $I_1$ and the $5 \times 5$ permutation matrix $Q = \begin{bmatrix}
0 & 0 & 0 & 1 & 0\\
1 & 0 & 0 & 0 & 0\\
0 & 0 & 1 & 0 & 0\\
0 & 0 & 0 & 0 & 1\\
0 & 1 & 0 & 0 & 0
\end{bmatrix}$. Note that any matrix $P$ that can be obtained from $Q$ by any sequence of vertical/horizontal reflections or $90$ degree rotations must also have $\sat(n, P) = O(1)$, since vertical/horizontal reflections and $90$ degree rotations do not change $\sat(n, P)$ or $\exm(n, P)$. Fulek and Keszegh also found infinite classes of 0-1 matrices $P$ for which $\sat(n, P) = \Theta(n)$, including 0-1 matrices $P$ in which every column has at least two ones, or every row has at least two ones. This result implies that $\sat(n, P) = \Theta(n)$ for almost all $k \times k$ 0-1 matrices $P$, since almost all $k \times k$ 0-1 matrices $P$ have at least two ones in every row and at least two ones in every column.

Fulek and Keszegh conjectured that there are many more 0-1 matrices $P$ such that $\sat(n, P) = O(1)$ besides $I_1$ and $Q$. They asked for a characterization of all permutation matrices $P$ for which $\sat(n, P) = O(1)$. We affirm Fulek and Keszegh's conjecture by proving that almost all $k \times k$ permutation matrices $P$ have $\sat(n, P) = O(1)$. In order to prove this result, we define a family of permutation matrices $P$ for which $\sat(n, P) = O(1)$, and then we prove that almost all $k \times k$ permutation matrices are in this family. We say that a permutation matrix $Q$ is \emph{ordinary} if no matrix that can be obtained from $Q$ by any sequence of vertical/horizontal reflections or $90$ degree rotations is in any of the following classes. We refer to the classes below as Class 1, Class 2, Class 3, and Class 4. We say that $Q$ \emph{reduces} to the Class $i$ if there is a sequence of vertical/horizontal reflections or $90$ degree rotations from $Q$ to an element of Class $i$. So, a permutation matrix $Q$ is ordinary if $Q$ does not reduce to any of the following classes.

\begin{enumerate}
\item Class 1 consists of permutation matrices $P$ for which the rows can be split into two nonempty sets $R_1$, $R_2$ of contiguous rows (where the rows in $R_1$ precede the rows in $R_2$) and the columns can be split into two nonempty sets $C_1$, $C_2$ of contiguous columns (where the columns in $C_1$ precede the columns in $C_2$) such that $P$ has ones in the submatrices restricted to $R_1 \times C_1$ and $R_2 \times C_2$, and $P$ only has ones in those submatrices except for at most $2$ ones in the submatrix restricted to $R_1 \times C_2$. If $P$ has two ones in the submatrix restricted to $R_1 \times C_2$, then they are in adjacent rows.

\item Class 2 consists of permutation matrices $P$ for which the rows can be split into three nonempty sets $R_1$, $R_2$, $R_3$ of contiguous rows (where the rows in $R_1$ precede the rows in $R_2$, which precede the rows in $R_3$) and the columns can be split into two nonempty sets $C_1$, $C_2$ of contiguous columns (where the columns in $C_1$ precede the columns in $C_2$) such that $P$ has ones in the submatrices restricted to $R_1 \times C_2$ and $R_3 \times C_2$, $P$ has at least two ones in the submatrix restricted to $R_2 \times C_1$, and $P$ only has ones in those submatrices except for at most $2$ ones in the submatrix restricted to $R_2 \times C_2$. If $P$ has two ones in the submatrix restricted to $R_2 \times C_2$, then they are in adjacent rows.

\item Class 3 consists of permutation matrices $P$ which can be obtained by starting with an element $X$ of Class $2$, taking the row $r$ of $X$ containing the rightmost one in the submatrix restricted to $R_2 \times C_1$, deleting $r$ from its current position, and adding $r$ in front of the first row of $X$.

\item Class 4 consists of permutation matrices $P$ for which the rows can be split into three nonempty sets $R_1$, $R_2$, $R_3$ of contiguous rows (where the rows in $R_1$ precede the rows in $R_2$, which precede the rows in $R_3$) and the columns can be split into three nonempty sets $C_1$, $C_2$, $C_3$ of contiguous columns (where the columns in $C_1$ precede the columns in $C_2$, which precede the columns in $C_3$) such that $P$ has ones in the submatrices restricted to $R_1 \times C_2$, $R_2 \times C_1$, $R_2 \times C_3$, and $R_3 \times C_2$, and only in those submatrices.

\end{enumerate}

There is clearly some overlap between the classes. For example, in the definition of Class 3, when we construct a permutation matrix $P$ in Class 3 from a permutation matrix $X$ in Class 2, the only instances when $P$ will not be an element of Class 2 are when $X$ has exactly two ones in the submatrix restricted to $R_2 \times C_1$. 

In the next section, we prove that every ordinary permutation matrix has bounded saturation function and almost all $k \times k$ permutation matrices are ordinary. In the final section, we discuss some possible directions for future research, including saturation functions of forbidden $d$-dimensional 0-1 matrices.

\section{New results}
 In this section, we will prove the main result below. 

\begin{thm}\label{mainth}
Almost all $k \times k$ permutation matrices $P$ have $\sat(n, P) = O(1)$. 
\end{thm}

In order to prove the main result, we first show that all ordinary permutation matrices have bounded saturation functions, and then we prove that almost all $k \times k$ permutation matrices are ordinary.

\begin{thm}
Every ordinary permutation matrix $P$ has $\sat(n, P) = O(1)$. 
\end{thm}

\begin{proof}
For every permutation matrix $P$, we define a 0-1 matrix $T_P$ with an odd number of rows and an odd number of columns such that the middle row and middle column of $T_P$ have all zero entries. We will prove that $T_P$ avoids $P$ for all ordinary permutation matrices $P$. We construct $T_P$ so that changing a zero to a one in the middle row or middle column of $T_P$ creates a copy of $P$.  In order to obtain a 0-1 matrix $T'_P$ from $T_P$ that is saturating for $P$, change zeroes in $T_P$ to ones one at a time without creating a copy of $P$ until it is impossible to change any more zeroes without creating a copy of $P$. By definition, the resulting matrix $T'_P$ must be saturating for $P$. Note that any zeroes in $T_P$ that we change to ones to make $T'_P$ must be off the middle row and off the middle column, or else it will create a copy of $P$.

In order to construct $T_P$, we define four submatrices of $P$: $P_N$ is obtained by deleting the row and column containing the bottom one in $P$, $P_S$ is obtained by deleting the row and column containing the top one in $P$, $P_E$ is obtained by deleting the row and column containing the leftmost one in $P$, and $P_W$ is obtained by deleting the row and column containing the rightmost one in $P$.

We start with $T_P$ as an $n \times n$ matrix of all zeroes with $n = 6k+1$. The exact value of $n$ in this proof is not important, just the fact that it is sufficiently large with respect to $k$. We place a copy of $P$ in the top $k$ rows of $T_P$, with the bottom one of this copy of $P$ in the middle column of $T_P$ and all columns in the copy adjacent, and then we delete the bottom one of $P$ in the copy to leave a copy of $P_N$ in the top $k-1$ rows of $T_P$. We place another copy of $P$ in the leftmost $k$ columns of $T_P$, with the rightmost one of this copy of $P$ in the middle row of $T_P$ and all rows in the copy adjacent, and then we delete the rightmost one of $P$ in the copy to leave a copy of $P_W$ in the leftmost $k-1$ columns of $T_P$.

We place another copy of $P$ in the bottom $k$ rows of $T_P$, with the top one of this copy of $P$ in the middle column of $T_P$. For any columns of $P$ adjacent to the column containing the top one, we put those columns in the bottom $k$ rows adjacent to the middle column of $T_P$. However, any columns of $P$ to the left of these columns go in the bottom $k$ rows directly to the left of the columns that contain the copy of $P_N$ in $T_P$, and any columns of $P$ to the right of these columns go in the bottom $k$ rows directly to the right of the columns that contain the copy of $P_N$ in $T_P$. After placing this copy of $P$ in the bottom $k$ rows of $T_P$, we delete the top one of $P$ in the copy to leave a copy of $P_S$ in the bottom $k-1$ rows of $T_P$. Note that if the top and bottom ones in $P$ are not the leftmost or rightmost ones in $P$, then the union of the columns containing $P_N$ and $P_S$ will consist of two sets of contiguous columns, separated by only the middle column of $T_P$, and the intersection of the columns containing $P_N$ and $P_S$ will consist only of the two columns that are adjacent to the middle column of $T_P$.

We place another copy of $P$ in the rightmost $k$ columns of $T_P$, with the leftmost one of $P$ in the middle row of $T_P$. For any rows of $P$ adjacent to the row containing the leftmost one, we put those rows in the rightmost $k$ columns adjacent to the middle row of $T_P$. However, any rows of $P$ above these rows go in the rightmost $k$ columns directly above the rows that contain the copy of $P_W$ in $T_P$, and any rows in $P$ under these rows go in the rightmost $k$ columns directly under the rows that contain the copy of $P_W$ in $T_P$. After placing this copy of $P$ in the rightmost $k$ columns of $T_P$, we delete the leftmost one of $P$ in the copy to leave a copy of $P_E$ in the rightmost $k-1$ columns of $T_P$. Note that if the leftmost and rightmost ones in $P$ are not the top or bottom ones in $P$, then the union of the rows containing $P_W$ and $P_E$ will consist of two sets of contiguous rows, separated by only the middle row of $T_P$, and the intersection of the rows containing $P_W$ and $P_E$ will consist only of the two rows that are adjacent to the middle row of $T_P$. We call the copies of $P_N, P_S, P_E$, and $P_W$ the \emph{sections} of $T_P$.  This completes the construction of $T_P$, and Figure \ref{fig:tq} shows $T_Q$ and $T_R$ for $Q = 
\begin{bmatrix}
0 & 0 & 0 & 1 & 0\\
1 & 0 & 0 & 0 & 0\\
0 & 0 & 1 & 0 & 0\\
0 & 0 & 0 & 0 & 1\\
0 & 1 & 0 & 0 & 0
\end{bmatrix}$ and $R = 
\begin{bmatrix}
0 & 0 & 0 & 1 & 0 & 0\\
1 & 0 & 0 & 0 & 0 & 0\\
0 & 0 & 1 & 0 & 0 & 0\\
0 & 0 & 0 & 0 & 0 & 1\\
0 & 1 & 0 & 0 & 0 & 0\\
0 & 0 & 0 & 0 & 1 & 0
\end{bmatrix}$.

\begin{figure}

  \begin{subfigure}[b]{0.3\textwidth}
     \tiny
\[
\begin{bmatrix*}[r]
0 & 0 & 0 & 0 & 0 & 0 & 0 & 0 & 0 & 0 & 0 & 0 & 0 & 0 & 0 & \textbf{0} & 0 & \textbf{1} & 0 & 0 & 0 & 0 & 0 & 0 & 0 & 0 & 0 & 0 & 0 & 0 & 0\\
0 & 0 & 0 & 0 & 0 & 0 & 0 & 0 & 0 & 0 & 0 & 0 & 0 & 0 & \textbf{1} & \textbf{0} & 0 & 0 & 0 & 0 & 0 & 0 & 0 & 0 & 0 & 0 & 0 & 0 & 0 & 0 & 0\\
0 & 0 & 0 & 0 & 0 & 0 & 0 & 0 & 0 & 0 & 0 & 0 & 0 & 0 & 0 & \textbf{0} & \textbf{1} & 0 & 0 & 0 & 0 & 0 & 0 & 0 & 0 & 0 & 0 & 0 & 0 & 0 & 0\\
0 & 0 & 0 & 0 & 0 & 0 & 0 & 0 & 0 & 0 & 0 & 0 & 0 & 0 & 0 & \textbf{0} & 0 & 0 & \textbf{1} & 0 & 0 & 0 & 0 & 0 & 0 & 0 & 0 & 0 & 0 & 0 & 0\\
0 & 0 & 0 & 0 & 0 & 0 & 0 & 0 & 0 & 0 & 0 & 0 & 0 & 0 & 0 & \textbf{0} & 0 & 0 & 0 & 0 & 0 & 0 & 0 & 0 & 0 & 0 & 0 & 0 & 0 & 0 & 0\\
0 & 0 & 0 & 0 & 0 & 0 & 0 & 0 & 0 & 0 & 0 & 0 & 0 & 0 & 0 & \textbf{0} & 0 & 0 & 0 & 0 & 0 & 0 & 0 & 0 & 0 & 0 & 0 & 0 & 0 & 0 & 0\\
0 & 0 & 0 & 0 & 0 & 0 & 0 & 0 & 0 & 0 & 0 & 0 & 0 & 0 & 0 & \textbf{0} & 0 & 0 & 0 & 0 & 0 & 0 & 0 & 0 & 0 & 0 & 0 & 0 & 0 & 0 & 0\\
0 & 0 & 0 & 0 & 0 & 0 & 0 & 0 & 0 & 0 & 0 & 0 & 0 & 0 & 0 & \textbf{0} & 0 & 0 & 0 & 0 & 0 & 0 & 0 & 0 & 0 & 0 & 0 & 0 & 0 & 0 & 0\\
0 & 0 & 0 & 0 & 0 & 0 & 0 & 0 & 0 & 0 & 0 & 0 & 0 & 0 & 0 & \textbf{0} & 0 & 0 & 0 & 0 & 0 & 0 & 0 & 0 & 0 & 0 & 0 & 0 & 0 & 0 & 0\\
0 & 0 & 0 & 0 & 0 & 0 & 0 & 0 & 0 & 0 & 0 & 0 & 0 & 0 & 0 & \textbf{0} & 0 & 0 & 0 & 0 & 0 & 0 & 0 & 0 & 0 & 0 & 0 & 0 & 0 & 0 & 0\\
0 & 0 & 0 & 0 & 0 & 0 & 0 & 0 & 0 & 0 & 0 & 0 & 0 & 0 & 0 & \textbf{0} & 0 & 0 & 0 & 0 & 0 & 0 & 0 & 0 & 0 & 0 & 0 & 0 & 0 & 0 & 0\\
0 & 0 & 0 & 0 & 0 & 0 & 0 & 0 & 0 & 0 & 0 & 0 & 0 & 0 & 0 & \textbf{0} & 0 & 0 & 0 & 0 & 0 & 0 & 0 & 0 & 0 & 0 & 0 & 0 & 0 & 0 & 0\\
0 & 0 & 0 & \textbf{1} & 0 & 0 & 0 & 0 & 0 & 0 & 0 & 0 & 0 & 0 & 0 & \textbf{0} & 0 & 0 & 0 & 0 & 0 & 0 & 0 & 0 & 0 & 0 & 0 & 0 & 0 & 0 & 0\\
\textbf{1} & 0 & 0 & 0 & 0 & 0 & 0 & 0 & 0 & 0 & 0 & 0 & 0 & 0 & 0 & \textbf{0} & 0 & 0 & 0 & 0 & 0 & 0 & 0 & 0 & 0 & 0 & 0 & 0 & 0 & 0 & 0\\
0 & 0 & \textbf{1} & 0 & 0 & 0 & 0 & 0 & 0 & 0 & 0 & 0 & 0 & 0 & 0 & \textbf{0} & 0 & 0 & 0 & 0 & 0 & 0 & 0 & 0 & 0 & 0 & 0 & 0 & 0 & \textbf{1} & 0\\
\textbf{0} & \textbf{0} & \textbf{0} & \textbf{0} & \textbf{0} & \textbf{0} & \textbf{0} & \textbf{0} & \textbf{0} & \textbf{0} & \textbf{0} & \textbf{0} & \textbf{0} & \textbf{0} & \textbf{0} & \textbf{0} & \textbf{0} & \textbf{0} & \textbf{0} & \textbf{0} & \textbf{0} & \textbf{0} & \textbf{0} & \textbf{0} &  \textbf{0} & \textbf{0} & \textbf{0} & \textbf{0} &  \textbf{0} & \textbf{0} & \textbf{0}\\
0 & \textbf{1} & 0 & 0 & 0 & 0 & 0 & 0 & 0 & 0 & 0 & 0 & 0 & 0 & 0 & \textbf{0} & 0 & 0 & 0 & 0 & 0 & 0 & 0 & 0 & 0 & 0 & 0 & 0 & \textbf{1} & 0 & 0\\
0 & 0 & 0 & 0 & 0 & 0 & 0 & 0 & 0 & 0 & 0 & 0 & 0 & 0 & 0 & \textbf{0} & 0 & 0 & 0 & 0 & 0 & 0 & 0 & 0 & 0 & 0 & 0 & 0 & 0 & 0 & \textbf{1}\\
0 & 0 & 0 & 0 & 0 & 0 & 0 & 0 & 0 & 0 & 0 & 0 & 0 & 0 & 0 & \textbf{0} & 0 & 0 & 0 & 0 & 0 & 0 & 0 & 0 & 0 & 0 & 0 & \textbf{1} & 0 & 0 & 0\\
0 & 0 & 0 & 0 & 0 & 0 & 0 & 0 & 0 & 0 & 0 & 0 & 0 & 0 & 0 & \textbf{0} & 0 & 0 & 0 & 0 & 0 & 0 & 0 & 0 & 0 & 0 & 0 & 0 & 0 & 0 & 0\\
0 & 0 & 0 & 0 & 0 & 0 & 0 & 0 & 0 & 0 & 0 & 0 & 0 & 0 & 0 & \textbf{0} & 0 & 0 & 0 & 0 & 0 & 0 & 0 & 0 & 0 & 0 & 0 & 0 & 0 & 0 & 0\\
0 & 0 & 0 & 0 & 0 & 0 & 0 & 0 & 0 & 0 & 0 & 0 & 0 & 0 & 0 & \textbf{0} & 0 & 0 & 0 & 0 & 0 & 0 & 0 & 0 & 0 & 0 & 0 & 0 & 0 & 0 & 0\\
0 & 0 & 0 & 0 & 0 & 0 & 0 & 0 & 0 & 0 & 0 & 0 & 0 & 0 & 0 & \textbf{0} & 0 & 0 & 0 & 0 & 0 & 0 & 0 & 0 & 0 & 0 & 0 & 0 & 0 & 0 & 0\\
0 & 0 & 0 & 0 & 0 & 0 & 0 & 0 & 0 & 0 & 0 & 0 & 0 & 0 & 0 & \textbf{0} & 0 & 0 & 0 & 0 & 0 & 0 & 0 & 0 & 0 & 0 & 0 & 0 & 0 & 0 & 0\\
0 & 0 & 0 & 0 & 0 & 0 & 0 & 0 & 0 & 0 & 0 & 0 & 0 & 0 & 0 & \textbf{0} & 0 & 0 & 0 & 0 & 0 & 0 & 0 & 0 & 0 & 0 & 0 & 0 & 0 & 0 & 0\\
0 & 0 & 0 & 0 & 0 & 0 & 0 & 0 & 0 & 0 & 0 & 0 & 0 & 0 & 0 & \textbf{0} & 0 & 0 & 0 & 0 & 0 & 0 & 0 & 0 & 0 & 0 & 0 & 0 & 0 & 0 & 0\\
0 & 0 & 0 & 0 & 0 & 0 & 0 & 0 & 0 & 0 & 0 & 0 & 0 & 0 & 0 & \textbf{0} & 0 & 0 & 0 & 0 & 0 & 0 & 0 & 0 & 0 & 0 & 0 & 0 & 0 & 0 & 0\\
0 & 0 & 0 & 0 & 0 & 0 & 0 & 0 & 0 & 0 & 0 & 0 & \textbf{1} & 0 & 0 & \textbf{0} & 0 & 0 & 0 & 0 & 0 & 0 & 0 & 0 & 0 & 0 & 0 & 0 & 0 & 0 & 0\\
0 & 0 & 0 & 0 & 0 & 0 & 0 & 0 & 0 & 0 & 0 & 0 & 0 & 0 & \textbf{1} & \textbf{0} & 0 & 0 & 0 & 0 & 0 & 0 & 0 & 0 & 0 & 0 & 0 & 0 & 0 & 0 & 0\\
0 & 0 & 0 & 0 & 0 & 0 & 0 & 0 & 0 & 0 & 0 & 0 & 0 & 0 & 0 & \textbf{0} & \textbf{1} & 0 & 0 & 0 & 0 & 0 & 0 & 0 & 0 & 0 & 0 & 0 & 0 & 0 & 0\\
0 & 0 & 0 & 0 & 0 & 0 & 0 & 0 & 0 & 0 & 0 & 0 & 0 & \textbf{1} & 0 & \textbf{0} & 0 & 0 & 0 & 0 & 0 & 0 & 0 & 0 & 0 & 0 & 0 & 0 & 0 & 0 & 0
\end{bmatrix*}
\]
         \caption{$T_Q$}
         \label{fig:y equals x}
     \end{subfigure}
     \hfill
  
  \begin{subfigure}[b]{0.2\textwidth}
       \tiny
\[
\begin{bmatrix*}[r]
0 & 0 & 0 & 0 & 0 & 0 & 0 & 0 & 0 & 0 & 0 & 0 & 0 & 0 & 0 & 0 & 0 & \textbf{1} & \textbf{0} & 0 & 0 & 0 & 0 & 0 & 0 & 0 & 0 & 0 & 0 & 0 & 0 & 0 & 0 & 0 & 0 & 0 & 0\\
0 & 0 & 0 & 0 & 0 & 0 & 0 & 0 & 0 & 0 & 0 & 0 & 0 & 0 & \textbf{1} & 0 & 0 & 0 & \textbf{0} & 0 & 0 & 0 & 0 & 0 & 0 & 0 & 0 &  0 & 0 & 0 & 0 & 0 & 0 & 0 & 0 & 0 & 0\\
0 & 0 & 0 & 0 & 0 & 0 & 0 & 0 & 0 & 0 & 0 & 0 & 0 & 0 & 0 & 0 & \textbf{1} & 0 & \textbf{0} & 0 & 0 & 0 & 0 & 0 & 0 & 0 & 0 & 0 & 0 & 0 & 0 & 0 & 0 & 0 & 0 & 0 & 0\\
0 & 0 & 0 & 0 & 0 & 0 & 0 & 0 & 0 & 0 & 0 & 0 & 0 & 0 & 0 & 0 & 0 & 0 & \textbf{0} & \textbf{1} & 0 & 0 & 0 & 0 & 0 & 0 & 0 & 0 & 0 & 0 & 0 & 0 & 0 & 0 & 0 & 0 & 0\\
0 & 0 & 0 & 0 & 0 & 0 & 0 & 0 & 0 & 0 & 0 & 0 & 0 & 0 & 0 & \textbf{1} & 0 & 0 & \textbf{0} & 0 & 0 & 0 & 0 & 0 & 0 & 0 & 0 & 0 & 0 & 0 & 0 & 0 & 0 & 0 & 0 & 0 & 0\\
0 & 0 & 0 & 0 & 0 & 0 & 0 & 0 & 0 & 0 & 0 & 0 & 0 & 0 & 0 & 0 & 0 & 0 & \textbf{0} & 0 & 0 & 0 & 0 & 0 & 0 & 0 & 0 & 0 & 0 & 0 & 0 & 0 & 0 & 0 & 0 & 0 & 0\\
0 & 0 & 0 & 0 & 0 & 0 & 0 & 0 & 0 & 0 & 0 & 0 & 0 & 0 & 0 & 0 & 0 & 0 & \textbf{0} & 0 & 0 & 0 & 0 & 0 & 0 & 0 & 0 & 0 & 0 & 0 & 0 & 0 & 0 & 0 & 0 & 0 & 0\\
0 & 0 & 0 & 0 & 0 & 0 & 0 & 0 & 0 & 0 & 0 & 0 & 0 & 0 & 0 & 0 & 0 & 0 & \textbf{0} & 0 & 0 & 0 & 0 & 0 & 0 & 0 & 0 & 0 & 0 & 0 & 0 & 0 & 0 & 0 & 0 & 0 & 0\\
0 & 0 & 0 & 0 & 0 & 0 & 0 & 0 & 0 & 0 & 0 & 0 & 0 & 0 & 0 & 0 & 0 & 0 & \textbf{0} & 0 & 0 & 0 & 0 & 0 & 0 & 0 & 0 & 0 & 0 & 0 & 0 & 0 & 0 & 0 & 0 & 0 & 0\\
0 & 0 & 0 & 0 & 0 & 0 & 0 & 0 & 0 & 0 & 0 & 0 & 0 & 0 & 0 & 0 & 0 & 0 & \textbf{0} & 0 & 0 & 0 & 0 & 0 & 0 & 0 & 0 & 0 & 0 & 0 & 0 & 0 & 0 & 0 & 0 & 0 & 0\\
0 & 0 & 0 & 0 & 0 & 0 & 0 & 0 & 0 & 0 & 0 & 0 & 0 & 0 & 0 & 0 & 0 & 0 & \textbf{0} & 0 & 0 & 0 & 0 & 0 & 0 & 0 & 0 & 0 & 0 & 0 & 0 & 0 & 0 & 0 & 0 & 0 & 0\\
0 & 0 & 0 & 0 & 0 & 0 & 0 & 0 & 0 & 0 & 0 & 0 & 0 & 0 & 0 & 0 & 0 & 0 & \textbf{0} & 0 & 0 & 0 & 0 & 0 & 0 & 0 & 0 & 0 & 0 & 0 & 0 & 0 & 0 & 0 & 0 & 0 & 0\\
0 & 0 & 0 & 0 & 0 & 0 & 0 & 0 & 0 & 0 & 0 & 0 & 0 & 0 & 0 & 0 & 0 & 0 & \textbf{0} & 0 & 0 & 0 & 0 & 0 & 0 & 0 & 0 & 0 & 0 & 0 & 0 & 0 & 0 & 0 & 0 & 0 & 0\\
0 & 0 & 0 & 0 & 0 & 0 & 0 & 0 & 0 & 0 & 0 & 0 & 0 & 0 & 0 & 0 & 0 & 0 & \textbf{0} & 0 & 0 & 0 & 0 & 0 & 0 & 0 & 0 & 0 & 0 & 0 & 0 & 0 & 0 & 0 & 0 & 0 & 0\\
0 & 0 & 0 & 0 & 0 & 0 & 0 & 0 & 0 & 0 & 0 & 0 & 0 & 0 & 0 & 0 & 0 & 0 & \textbf{0} & 0 & 0 & 0 & 0 & 0 & 0 & 0 & 0 & 0 & 0 & 0 & 0 & 0 & 0 & 0 & 0 & 0 & 0\\
0 & 0 & 0 & \textbf{1} & 0 & 0 & 0 & 0 & 0 & 0 & 0 & 0 & 0 & 0 & 0 & 0 & 0 & 0 & \textbf{0} & 0 & 0 & 0 & 0 & 0 & 0 & 0 & 0 & 0 & 0 & 0 & 0 & 0 & 0 & 0 & 0 & 0 & 0\\
\textbf{1} & 0 & 0 & 0 & 0 & 0 & 0 & 0 & 0 & 0 & 0 & 0 & 0 & 0 & 0 & 0 & 0 & 0 & \textbf{0} & 0 & 0 & 0 & 0 & 0 & 0 & 0 & 0 & 0 & 0 & 0 & 0 & 0 & 0 & 0 & 0 & 0 & 0\\
0 & 0 & \textbf{1} & 0 & 0 & 0 & 0 & 0 & 0 & 0 & 0 & 0 & 0 & 0 & 0 & 0 & 0 & 0 & \textbf{0} & 0 & 0 & 0 & 0 & 0 & 0 & 0 & 0 & 0 & 0 & 0 & 0 & 0 & 0 & 0 & \textbf{1} & 0 & 0\\
\textbf{0} & \textbf{0} & \textbf{0} & \textbf{0} & \textbf{0} & \textbf{0} & \textbf{0} & \textbf{0} & \textbf{0} & \textbf{0} & \textbf{0} & \textbf{0} & \textbf{0} & \textbf{0} & \textbf{0} & \textbf{0} & \textbf{0} & \textbf{0} & \textbf{0} & \textbf{0} & \textbf{0} & \textbf{0} & \textbf{0} & \textbf{0} &  \textbf{0} & \textbf{0} & \textbf{0}& \textbf{0} & \textbf{0} & \textbf{0} & \textbf{0} & \textbf{0} & \textbf{0} & \textbf{0} &  \textbf{0} & \textbf{0} & \textbf{0}\\
0 & \textbf{1} & 0 & 0 & 0 & 0 & 0 & 0 & 0 & 0 & 0 & 0 & 0 & 0 & 0 & 0 & 0 & 0 & \textbf{0} & 0 & 0 & 0 & 0 & 0 & 0 & 0 & 0 & 0 & 0 & 0 & 0 & 0 & 0 & \textbf{1} & 0 & 0 & 0\\
0 & 0 & 0 & 0 & \textbf{1} & 0 & 0 & 0 & 0 & 0 & 0 & 0 & 0 & 0 & 0 & 0 & 0 & 0 & \textbf{0} & 0 & 0 & 0 & 0 & 0 & 0 & 0 & 0 & 0 & 0 & 0 & 0 & 0 & 0 & 0 & 0 & 0 & 0\\
0 & 0 & 0 & 0 & 0 & 0 & 0 & 0 & 0 & 0 & 0 & 0 & 0 & 0 & 0 & 0 & 0 & 0 & \textbf{0} & 0 & 0 & 0 & 0 & 0 & 0 & 0 & 0 & 0 & 0 & 0 & 0 & 0 & 0 & 0 & 0 & 0 & \textbf{1}\\
0 & 0 & 0 & 0 & 0 & 0 & 0 & 0 & 0 & 0 & 0 & 0 & 0 & 0 & 0 & 0 & 0 & 0 & \textbf{0} & 0 & 0 & 0 & 0 & 0 & 0 & 0 & 0 & 0 & 0 & 0 & 0 & 0 & \textbf{1} & 0 & 0 & 0 & 0\\
0 & 0 & 0 & 0 & 0 & 0 & 0 & 0 & 0 & 0 & 0 & 0 & 0 & 0 & 0 & 0 & 0 & 0 & \textbf{0} & 0 & 0 & 0 & 0 & 0 & 0 & 0 & 0 & 0 & 0 & 0 & 0 & 0 & 0 & 0 & 0 & \textbf{1} & 0\\
0 & 0 & 0 & 0 & 0 & 0 & 0 & 0 & 0 & 0 & 0 & 0 & 0 & 0 & 0 & 0 & 0 & 0 & \textbf{0} & 0 & 0 & 0 & 0 & 0 & 0 & 0 & 0 & 0 & 0 & 0 & 0 & 0 & 0 & 0 & 0 & 0 & 0\\
0 & 0 & 0 & 0 & 0 & 0 & 0 & 0 & 0 & 0 & 0 & 0 & 0 & 0 & 0 & 0 & 0 & 0 & \textbf{0} & 0 & 0 & 0 & 0 & 0 & 0 & 0 & 0 & 0 & 0 & 0 & 0 & 0 & 0 & 0 & 0 & 0 & 0\\
0 & 0 & 0 & 0 & 0 & 0 & 0 & 0 & 0 & 0 & 0 & 0 & 0 & 0 & 0 & 0 & 0 & 0 & \textbf{0} & 0 & 0 & 0 & 0 & 0 & 0 & 0 & 0 & 0 & 0 & 0 & 0 & 0 & 0 & 0 & 0 & 0 & 0\\
0 & 0 & 0 & 0 & 0 & 0 & 0 & 0 & 0 & 0 & 0 & 0 & 0 & 0 & 0 & 0 & 0 & 0 & \textbf{0} & 0 & 0 & 0 & 0 & 0 & 0 & 0 & 0 & 0 & 0 & 0 & 0 & 0 & 0 & 0 & 0 & 0 & 0\\
0 & 0 & 0 & 0 & 0 & 0 & 0 & 0 & 0 & 0 & 0 & 0 & 0 & 0 & 0 & 0 & 0 & 0 & \textbf{0} & 0 & 0 & 0 & 0 & 0 & 0 & 0 & 0 & 0 & 0 & 0 & 0 & 0 & 0 & 0 & 0 & 0 & 0\\
0 & 0 & 0 & 0 & 0 & 0 & 0 & 0 & 0 & 0 & 0 & 0 & 0 & 0 & 0 & 0 & 0 & 0 & \textbf{0} & 0 & 0 & 0 & 0 & 0 & 0 & 0 & 0 & 0 & 0 & 0 & 0 & 0 & 0 & 0 & 0 & 0 & 0\\
0 & 0 & 0 & 0 & 0 & 0 & 0 & 0 & 0 & 0 & 0 & 0 & 0 & 0 & 0 & 0 & 0 & 0 & \textbf{0} & 0 & 0 & 0 & 0 & 0 & 0 & 0 & 0 & 0 & 0 & 0 & 0 & 0 & 0 & 0 & 0 & 0 & 0\\
0 & 0 & 0 & 0 & 0 & 0 & 0 & 0 & 0 & 0 & 0 & 0 & 0 & 0 & 0 & 0 & 0 & 0 & \textbf{0} & 0 & 0 & 0 & 0 & 0 & 0 & 0 & 0 & 0 & 0 & 0 & 0 & 0 & 0 & 0 & 0 & 0 & 0\\
0 & 0 & 0 & 0 & 0 & 0 & 0 & 0 & 0 & 0 & 0 & 0 & \textbf{1} & 0 & 0 & 0 & 0 & 0 & \textbf{0} & 0 & 0 & 0 & 0 & 0 & 0 & 0 & 0 & 0 & 0 & 0 & 0 & 0 & 0 & 0 & 0 & 0 & 0\\
0 & 0 & 0 & 0 & 0 & 0 & 0 & 0 & 0 & 0 & 0 & 0 & 0 & 0 & 0 & 0 & 0 & \textbf{1} & \textbf{0} & 0 & 0 & 0 & 0 & 0 & 0 & 0 & 0 & 0 & 0 & 0 & 0 & 0 & 0 & 0 & 0 & 0 & 0\\
0 & 0 & 0 & 0 & 0 & 0 & 0 & 0 & 0 & 0 & 0 & 0 & 0 & 0 & 0 & 0 & 0 & 0 & \textbf{0} & 0 & \textbf{1} & 0 & 0 & 0 & 0 & 0 & 0 & 0 & 0 & 0 & 0 & 0 & 0 & 0 & 0 & 0 & 0\\
0 & 0 & 0 & 0 & 0 & 0 & 0 & 0 & 0 & 0 & 0 & 0 & 0 & \textbf{1} & 0 & 0 & 0 & 0 & \textbf{0} & 0 & 0 & 0 & 0 & 0 & 0 & 0 & 0 & 0 & 0 & 0 & 0 & 0 & 0 & 0 & 0 & 0 & 0\\
0 & 0 & 0 & 0 & 0 & 0 & 0 & 0 & 0 & 0 & 0 & 0 & 0 & 0 & 0 & 0 & 0 & 0 & \textbf{0} & \textbf{1} & 0 & 0 & 0 & 0 & 0 & 0 & 0 & 0 & 0 & 0 & 0 & 0 & 0 & 0 & 0 & 0 & 0
\end{bmatrix*}
\]

         \caption{$T_R$}
         \label{fig:five over x}
     \end{subfigure}
   
        \caption{$T_Q$ and $T_R$, with middle rows and middle columns and all ones in bold}
        \label{fig:tq}
\end{figure}

\normalsize

We need to prove that $T_P$ avoids $P$ if $P$ is ordinary. Suppose that $T_P$ contains $P$, in order to show that $P$ is not ordinary. It is impossible for the copy of $P$ to be fully contained in a single section of $T_P$, since each section has too few ones. Thus the copy of $P$ must be contained in at least two sections of $T_P$. If the copy of $P$ is contained in exactly two sections of $T_P$, then there are six possibilities. If $P$ is contained in $P_N$ and $P_E$, $P_E$ and $P_S$, $P_S$ and $P_W$, or $P_W$ and $P_N$, then $P$ reduces to Class 1, so $P$ is not ordinary.

Suppose that the copy of $P$ is on $P_W$ and $P_E$ (the proof is analogous for $P_N$ and $P_S$). If the leftmost or rightmost one in $P$ is also the top or bottom one in $P$, then $P$ reduces to Class 1, so $P$ is not ordinary. So, we may assume that the leftmost and rightmost ones in $P$ are not the top or bottom ones in $P$. If the copy of $P$ only had a single one in $P_E$ (resp. $P_W$), then the copy of $P$ would have to use all of the ones in $P_W$ (resp. $P_E$). However, the leftmost or rightmost one would be out of position by how we constructed $T_P$, if the leftmost and rightmost ones in $P$ are not the top or bottom ones in $P$. So there is no copy of $P$ on $P_W$ and $P_E$ with only a single one in $P_E$ or $P_W$, if the leftmost and rightmost ones in $P$ are not the top or bottom ones in $P$. It remains to consider the cases when the copy of $P$ on $P_W$ and $P_E$ has multiple ones in both $P_W$ and $P_E$. If the copy of $P$ has ones in $P_E$ above the rows that contain $P_W$ and below the rows that contain $P_W$, then $P$ is a Class 2 permutation matrix, so $P$ is not ordinary. If the copy of $P$ has ones in $P_E$ above the rows that contain $P_W$, but not below the rows that contain $P_W$, then $P$ reduces to Class 1, so $P$ is not ordinary. If the copy of $P$ has ones in $P_E$ below the rows that contain $P_W$, but not above the rows that contain $P_W$, then $P$ is in Class 1, so $P$ is not ordinary. Finally if the copy of $P$ had ones in $P_E$ only in the rows that contain $P_W$, then the copy of $P$ would only have the two ones in $P_E$ that are adjacent to the middle row of $T_P$, so $P$ would reduce to Class 1 or Class 2. However, two of the ones in $P_W$ are adjacent to the middle row of $T_P$, so there are not enough remaining ones in $P_W$ to form a copy of $P$. So, there is no copy of $P$ on $P_W$ and $P_E$ with ones only in the rows that contain $P_W$. 

If the copy of $P$ is contained in exactly $3$ sections (which will always have two opposite sections and a \emph{middle} section between them), then $P$ reduces to a Class 2 permutation matrix unless the copy of $P$ only has a single one in the middle section. Suppose that the copy of $P$ only has a single one in the middle section. Let $P'$ be the matrix obtained from $P$ by removing the row and column containing the one in the middle section. Suppose that the copy of $P'$ is on $P_W$ and $P_E$ (the proof is analogous for $P_N$ and $P_S$). If the copy of $P'$ only has a single one in $P_E$ or $P_W$, then $P$ must reduce to Class 1 or Class 3. Now suppose the copy of $P'$ has multiple ones in both $P_W$ and $P_E$. If the copy of $P'$ has ones in $P_E$ above the rows that contain $P_W$ and below the rows that contain $P_W$, then $P'$ and $P$ are in Class 2, so $P$ is not ordinary. If the copy of $P'$ has ones in $P_E$ above the rows that contain $P_W$, but not below the rows that contain $P_W$, then $P'$ and $P$ reduce to Class 1, so $P$ is not ordinary. If the copy of $P'$ has ones in $P_E$ below the rows that contain $P_W$, but not above the rows that contain $P_W$, then $P'$ and $P$ are in Class 1, so $P$ is not ordinary. If the copy of $P'$ has ones in $P_E$ only in the rows that contain $P_W$, then the copy of $P'$ only has the two ones in $P_E$ that are adjacent to the middle row of $T_P$, so both $P'$ and $P$ reduce to Class 1 or Class 2. 

Finally if the copy of $P$ has ones in all four sections, then $P$ is a Class 4 permutation matrix. Thus, in all possible cases $P$ is not ordinary.

In order to see that $\sat(n, P) = O(1)$ for every ordinary permutation matrix $P$, observe that changing a zero to a one in the middle row or the middle column of $T'_P$ will create a copy of $P$. For any integer $j > 0$, we can add $j$ rows of zeroes on the middle row of $T'_P$ and $j$ columns of zeroes on the middle column of $T'_P$. The resulting 0-1 matrix must still be saturating for $P$.
\end{proof}

The next lemma completes the proof of Theorem \ref{mainth}.

\begin{lem}\label{almostall}
Almost all $k \times k$ permutation matrices $P$ are ordinary.
\end{lem}

\begin{proof}
It suffices to prove that the number of permutation matrices in Classes 1, 2, 3, and 4 is $o(k!)$. For Class 1, we split the bound into three parts. First, consider Class 1 permutation matrices with no ones in the submatrix restricted to $R_1 \times C_2$. The number of these matrices with $|R_1| = j$ is $j! (k-j)!$ since there are $j!$ possibilities for the entries in $R_1 \times C_1$ and $(k-j)!$ possibilities for the entries in $R_2 \times C_2$. When $j = 1$ or $j = k-1$, $j!(k-j)! = (k-1)! = o(k!)$. When $2 \le j \le k-2$, $j!(k-j)! = O(\frac{k!}{k^2})$, so $\sum_{j = 2}^{k-2} j!(k-j)! = o(k!)$. Thus the total number of permutation matrices in Class 1 with no ones in the submatrix restricted to $R_1 \times C_2$ is $o(k!)$.

Next, consider Class 1 permutation matrices with a single one in the submatrix restricted to $R_1 \times C_2$. The number of these matrices with $|R_1| = j+1$ is at most $(j+1) j! (k-j)!$. We must have $j \le k-2$ since $j+1 < k$. When $j = 1$ or $j = k-2$, $(j+1)j!(k-j)! = o(k!)$. When $2 \le j \le k-3$, $(j+1)j!(k-j)! = O(\frac{k!}{k^2})$, so $\sum_{j = 2}^{k-3} (j+1)j!(k-j)! = o(k!)$. Thus the total number of permutation matrices in Class 1 with a single one in the submatrix restricted to $R_1 \times C_2$ is $o(k!)$.

Next, consider Class 1 permutation matrices with two ones in the submatrix restricted to $R_1 \times C_2$. The number of these matrices with $|R_1| = j+2$ is at most $(j+1) j! (k-j)!$. We must have $j \le k-3$ since $j+2 < k$. When $j = 1$, $(j+1)j!(k-j)! = o(k!)$. When $2 \le j \le k-3$, $(j+1)j!(k-j)! = O(\frac{k!}{k^2})$, so $\sum_{j = 2}^{k-3} (j+1)j!(k-j)! = o(k!)$. Thus the total number of permutation matrices in Class 1 with two ones in the submatrix restricted to $R_1 \times C_2$ is $o(k!)$.

For Class 2, we again split the bound into $3$ parts. First we bound the number of permutation matrices in Class 2 with no ones in the submatrix restricted to $R_2 \times C_2$. When $|R_1| = i$ and $|R_2| = j$, there are at most $j!(k-j)!$ permutation matrices in Class 2 with no ones in the submatrix restricted to $R_2 \times C_2$. By definition of Class 2, we must have $2 \le j \le k-2$. When $j = 2$ or $j = k-2$, $j!(k-j)! = 2 (k-2)! = O(\frac{k!}{k^2})$, so $k j!(k-j)! = o(k!)$, covering all possible $i$ for $j = 2$ and $j = k-2$ since $i < k$. When $3 \le j \le k-3$, $j!(k-j)! = O(\frac{k!}{k^3})$, so $\sum_{j = 3}^{k-3} j!(k-j)! = O(\frac{k!}{k^2})$ and $k \sum_{j = 3}^{k-3} j!(k-j)! = o(k!)$, covering all possible $i$ for $3 \le j \le k-3$. Thus the total number of permutation matrices in Class 2 with no ones in the submatrix restricted to $R_2 \times C_2$ is $o(k!)$.

Next we bound the number of permutation matrices in Class 2 with a single one in the submatrix restricted to $R_2 \times C_2$. When $|R_1| = i$ and $|R_2| = j+1$, there are at most $(j+1) j!(k-j)!$ permutation matrices in Class 2 with a single one in the submatrix restricted to $R_2 \times C_2$. By definition of Class 2, we must have $j \ge 2$. When $j = 2$, $(j+1) j!(k-j)! = O(\frac{k!}{k^2})$, so $k (j+1) j!(k-j)! = o(k!)$, which covers all possible $i$ for $j = 2$. We cannot have $j = k-2$, since $i > 0$ and $i + (j+1) < k$.  When $j = k-3$, $(j+1) j!(k-j)! = o(k!)$ and $i = 1$ since $i > 0$ and $i + (j+1) < k$. When $3 \le j \le k-4$, $(j+1)j!(k-j)! = O(\frac{k!}{k^3})$, so $\sum_{j = 3}^{k-4} (j+1)j!(k-j)! = O(\frac{k!}{k^2})$ and $k \sum_{j = 3}^{k-4} (j+1)j!(k-j)! = o(k!)$, which covers all possible $i$ for $3 \le j \le k-4$. Thus the total number of permutation matrices in Class 2 with a single one in the submatrix restricted to $R_2 \times C_2$ is $o(k!)$.

Next we bound the number of permutation matrices in Class 2 with two ones in the submatrix restricted to $R_2 \times C_2$. When $|R_1| = i$ and $|R_2| = j+2$, there are at most $(j+1) j!(k-j)!$ permutation matrices in Class 2 with two ones in the submatrix restricted to $R_2 \times C_2$. By definition of Class 2, we must have $j \ge 2$. When $j = 2$, $(j+1) j!(k-j)! = O(\frac{k!}{k^2})$, so $k (j+1) j!(k-j)! = o(k!)$, which covers all possible $i$ for $j = 2$. We cannot have $j = k-2$ or $j = k-3$, since $i > 0$ and $i+(j+2) < k$. When $3 \le j \le k-4$, $(j+1) j!(k-j)! = O(\frac{k!}{k^3})$, so $\sum_{j = 3}^{k-4} (j+1) j!(k-j)! = O(\frac{k!}{k^2})$ and $k \sum_{j = 3}^{k-4} (j+1) j!(k-j)! = o(k!)$, covering all possible $i$ for $3 \le j \le k-4$. Thus the total number of permutation matrices in Class 2 with two ones in the submatrix restricted to $R_2 \times C_2$ is $o(k!)$.

Note that each permutation matrix in Class 3 is created from at least one permutation matrix in Class 2, and each permutation matrix in Class 2 creates only one permutation matrix in Class 3, so the number of permutation matrices in Class 3 is $o(k!)$.

Finally we bound the number of permutation matrices in Class 4. When $|R_2| = j$ for a permutation matrix in Class 4, we must have $|C_2| = k-j$, so the number of permutation matrices in Class 4 with $|R_1| = i$, $|R_2| = j$, and $|C_1| = r$ is at most $j!(k-j)!$. We must have $2 \le |R_2| \le k-2$ for a permutation matrix in Class 4, since both the submatrix restricted to $R_2 \times C_1$ and the submatrix restricted to $R_2 \times C_3$ must have ones, as must the submatrices restricted to $R_1 \times C_2$ and $R_3 \times C_2$. If $j = 2$, then $|C_2| = k-2$, so $r = 1$ and the number of permutation matrices in Class 4 with $|R_2| = 2$ is at most $k 2!(k-2)! = o(k!)$, covering all possible $i$ for $j = 2$. If $j = k-2$, then $i = 1$, so the number of permutation matrices in Class 4 with $|R_2| = k-2$ is at most $k 2!(k-2)! = o(k!)$, covering all possible $r$ for $j = k-2$. If $j = 3$, then $|C_2| = k-3$, so $r = 1$ or $r = 2$ and the number of permutation matrices in Class 4 with $|R_2| = 3$ is at most $2k 3!(k-3)! = o(k!)$, covering all possible $i$ and $r$ for $j = 3$. If $j = k-3$, then $i = 1$ or $i = 2$, so the number of permutation matrices in Class 4 with $|R_2| = k-3$ is at most $2k 3!(k-3)! = o(k!)$, covering all possible $i$ and $r$ for $j = k-3$. If $4 \le j \le k-4$, then $j!(k-j)! = O(\frac{k!}{k^4})$, so $\sum_{j = 4}^{k-4} j! (k-j)! = O(\frac{k!}{k^3})$ and $k^2 \sum_{j = 4}^{k-4} j! (k-j)! = o(k!)$, covering all possible $i$ and $r$ for $4 \le j \le k-4$. Thus the total number of permutation matrices in Class 4 is $o(k!)$.
\end{proof}

\section{Discussion}

We showed that almost all $k \times k$ permutation matrices have bounded saturation functions, but Fulek and Keszegh's problem of characterizing all permutation matrices $P$ for which $\sat(n, P) = O(1)$ is still open. We made some progress on this problem by showing that every ordinary permutation matrix $P$ has $\sat(n, P) = O(1)$. Ordinary permutation matrices are the permutation matrices that do not reduce to any of Class 1, Class 2, Class 3, or Class 4, so the remaining problem is to characterize the permutation matrices $P$ in Classes 1, 2, 3, and 4 for which $\sat(n, P) = O(1)$. Fulek and Keszegh showed that all permutation matrices $P$ in Class 1 with no ones in the submatrix restricted to $R_1 \times C_2$ have $\sat(n, P) = \Theta(n)$ \cite{fk}. What about the rest of Class 1, Class 2, Class 3, and Class 4?

As for specific 0-1 matrices, we have not determined whether $\sat(n, P) = O(1)$ for $P = 
\begin{bmatrix}
0 & 0  & 1 & 0\\
1 & 0 & 0 & 0\\
0 & 0 & 0 & 1\\
0 & 1 & 0 & 0
\end{bmatrix}$. This was another open problem from \cite{fk}. Observe that $P$ is the matrix obtained by removing the middle row and middle column of $Q$. In this case we may construct $T_P$, but $P$ is in Class 4 and $T_P$ contains $P$.

Let $L_k$ denote the number of $k \times k$ permutation matrices $P$ for which $\sat(n, P) = \Theta(n)$. The proof of Lemma \ref{almostall} implies that $L_k = O((k-1)!)$. We also have $L_k = \Omega((k-1)!)$, since Fulek and Keszegh showed that $\sat(n, P) = \Theta(n)$ for all permutation matrices $P$ in Class 1 with no ones in the submatrix restricted to $R_1 \times C_2$. Thus, $L_k = \Theta((k-1)!)$. It would be interesting to sharpen the bounds on $L_k$.

Our construction implies that $\sat(n, P) = O(k^2)$ for any ordinary $k \times k$ permutation matrix $P$. Another interesting problem would be to bound the maximum possible value of $\sat(n, P)$ for any $k \times k$ permutation matrix $P$ with $\sat(n, P) = O(1)$. A more general open problem from Fulek and Keszegh is to determine whether there exists a decision procedure to answer whether $\sat(n, P) = O(1)$ for any given 0-1 matrix $P$ \cite{fk}. An approach would be to bound the maximum possible value of $\sat(n, P)$ for any $k \times k$ 0-1 matrix $P$ with $\sat(n, P) = O(1)$. It would also be interesting to investigate the complexity of computing $\sat(n, P)$ and determining whether $\sat(n, P) \leq m$, both in general and for special families of 0-1 matrices $P$ such as permutation matrices.

Finally, a natural direction for future research is to investigate saturation functions of forbidden $d$-dimensional 0-1 matrices. Extremal functions of forbidden $d$-dimensional 0-1 matrices have been investigated in \cite{ck, ff0, gt, km, gkst}. As with the $2$-dimensional case, we say that a $d$-dimensional 0-1 matrix $A$ is \emph{saturating} for the forbidden $d$-dimensional 0-1 matrix $P$ if $A$ avoids $P$ but changing any zero to a one in $A$ creates a copy of $P$. For any $d$-dimensional 0-1 matrix $P$, we define the saturation function $\sat(n, P, d)$ to be the minimum possible number of ones in a $d$-dimensional 0-1 matrix of dimensions $n \times \dots \times n$ that is saturating for $P$. 

Given any $2$-dimensional 0-1 matrix $P$ of dimensions $r \times s$, we can create a $d$-dimensional 0-1 matrix $P_d$ corresponding to $P$ for which all dimensions are of length $1$ except the first and second dimensions which are of length $r$ and $s$ respectively. We define entry $(i, j, 1, \dots, 1)$ of $P_d$ to be $1$ if and only if entry $(i, j)$ of $P$ is $1$. We have $\sat(n, P_d, d) = n^{d-2} \sat(n, P)$. To see that $\sat(n, P_d, d) \le n^{d-2} \sat(n, P)$, consider any $n \times n$ 0-1 matrix $A$ with $\sat(n, P)$ ones that is saturating for $P$, and let $B$ be the $d$-dimensional 0-1 matrix of dimensions $n \times \dots \times n$ obtained from $A$ by defining entry $(x_1, x_2, x_3, \dots, x_d)$ of $B$ to be $1$ if and only if entry $(x_1, x_2)$ of $A$ is $1$. Clearly $B$ avoids $P_d$, since a copy of $P_d$ in $B$ would imply there is a copy of $P$ in $A$. Moreover if we change any zero in entry $(x_1, x_2, x_3, \dots ,x_d)$ of $B$ to a one, then we create a copy of $P_d$ in $B$ for which the last $d-2$ coordinates of all entries in the copy are $x_3, \dots, x_d$. This is because the entries of $B$ with last $d-2$ coordinates equal to $x_3, \dots, x_d$ form a copy of $A$ when restricted only to the first two dimensions, and changing any zero in $A$ to a one creates a copy of $P$. Thus $B$ is saturating for $P_d$, and $B$ has $n^{d-2} \sat(n, P)$ ones, so $\sat(n, P_d, d) \le n^{d-2} \sat(n, P)$.

To see that $\sat(n, P_d, d) \ge n^{d-2} \sat(n, P)$, suppose for contradiction that $\sat(n, P_d, d) < n^{d-2} \sat(n, P)$. Then there exists a $d$-dimensional 0-1 matrix $C$ of dimensions $n \times \dots \times n$ that avoids $P_d$ with fewer than $n^{d-2} \sat(n, P)$ ones such that changing any zero in $C$ to a one creates a copy of $P_d$. By the pigeonhole principle, there exist $x_3, \dots, x_d$ such that the entries of $C$ with last $d-2$ coordinates equal to $x_3, \dots, x_d$ contain fewer than $\sat(n, P)$ ones. Let $T$ be the $2$-dimensional 0-1 matrix for which entry $(i, j)$ of $T$ is $1$ if and only if entry $(i, j, x_3, \dots, x_d)$ of $C$ is $1$. Then $T$ is an $n \times n$ 0-1 matrix that is saturating for $P$, but $T$ has fewer than $\sat(n, P)$ ones, a contradiction. Thus we proved the following lemma.

\begin{lem}
Given any $2$-dimensional 0-1 matrix $P$ of dimensions $r \times s$, let $P_d$ be the $d$-dimensional 0-1 matrix of dimensions $r \times s \times 1 \times \dots \times 1$ for which entry $(i, j, 1, \dots, 1)$ of $P_d$ is $1$ if and only if entry $(i, j)$ of $P$ is $1$. Then $\sat(n, P_d, d) = n^{d-2} \sat(n, P)$.
\end{lem}

It would be interesting to investigate the saturation functions of $d$-dimensional 0-1 matrices for which more than two of the dimensions have length greater than $1$. For example, what is $\sat(n, P, d)$ for every $d$-dimensional permutation matrix $P$?


\end{document}